\documentclass{amsart}[12pt]
\def\doctype{}

\usepackage{latexsym,amssymb,dsfont}
\usepackage{tikz}
\usetikzlibrary { decorations.pathmorphing, decorations.pathreplacing, decorations.shapes, }
\usepackage{caption}
\usepackage{color}
\usepackage{fancyhdr}
\usepackage{hyperref}
\usepackage{amsmath}
\hypersetup{
colorlinks=true,
citecolor=blue,
linkcolor=blue,
  urlcolor=blue
}

\newcommand{\comment}[1]{}

\numberwithin{equation}{section}

\fancypagestyle{titlepage}{
\fancyhead[R]{\doctype}
\fancyhead[CL]{}
\cfoot{\vspace{5pt} \thepage}
}

\let\oldsection\section
\newcommand\boldsection[1]{\oldsection{\bf #1}}
\newcommand\starsection[1]{\oldsection*{\bf #1}}
\makeatletter
\renewcommand\section{\@ifstar\starsection\boldsection}
\makeatother

\newtheoremstyle{theorem}
  {12pt}		  % space above
  {0pt}  % space below
  {\sl}  % bofy font
  {\parindent}     % ident - empty=no indent,  \parindent= paragraph indent
  {\bf}  % thm head font
  {. }    % punctuation after thm head
  { }    % space after thm head: `` ``=normal \newline=linebreak
  {}     % thm head specification
\theoremstyle{theorem}
\newtheorem{thm}{Theorem}[section]  % 1st argument is your name for it
\newtheorem{lemma}[thm]{Lemma}     % 2nd argument is what is printed

\newtheorem{cons}[thm]{Construction}

\newtheoremstyle{definition}
  {12pt}		  % space above
  {0pt}  % space below
  {}  % bofy font
  {\parindent}     % ident - empty=no indent,  \parindent= paragraph indent
  {\bf}  % thm head font
  {. }    % punctuation after thm head
  { }    % space after thm head: `` ``=normal \newline=linebreak
  {}     % thm head specification
\theoremstyle{definition}

\newtheorem{ex}[thm]{Example}

\renewcommand{\proofname}{Proof}

\makeatletter
\renewenvironment{proof}[1][\proofname]{\par
  \pushQED{\qed}%
  \normalfont \partopsep=\z@skip \topsep=\z@skip
  \trivlist
  \item[\hskip\labelsep
        \scshape
    #1\@addpunct{.}]\ignorespaces
}{%
  \popQED\endtrivlist\@endpefalse
}
\makeatother

\makeatletter
\renewcommand*\@maketitle{%
  \normalfont\normalsize
  \@adminfootnotes
  \@mkboth{\@nx\shortauthors}{\@nx\shorttitle}%
  \global\topskip42\p@\relax % 5.5pc   "   "   "     "     "
  \@settitle
  \ifx\@empty\authors \else {\vskip 1em
\vtop{\centering\shortauthors\@@par}} \fi
  \ifx\@empty\@date \else {\vskip 1em \vtop{\centering\@date\@@par}}\fi % MYCHANGE
  \ifx\@empty\@dedicatory
  \else
    \baselineskip18\p@
    \vtop{\centering{\footnotesize\itshape\@dedicatory\@@par}%
      \global\dimen@i\prevdepth}\prevdepth\dimen@i
  \fi
  \@setabstract
  \normalsize
  \if@titlepage
    \newpage
  \else
    \dimen@34\p@ \advance\dimen@-\baselineskip
    \vskip\dimen@\relax
  \fi
} % end \@maketitle
\renewcommand*\@adminfootnotes{%
  \let\@makefnmark\relax  \let\@thefnmark\relax
  \ifx\@empty\@subjclass\else \@footnotetext{\@setsubjclass}\fi
  \ifx\@empty\@keywords\else \@footnotetext{\@setkeywords}\fi
  \ifx\@empty\thankses\else \@footnotetext{%
    \def\par{\let\par\@par}\@setthanks}%
  \fi
\thispagestyle{titlepage}
}
\makeatother

\makeatletter \let\c@table\c@figure \makeatother

\newtheoremstyle{claimstyle}{4pt}{0pt}{}{15pt}{\bfseries}{.}{.5em}{}
\theoremstyle{claimstyle}

\begin{document}

\title[]{\large Intersection numbers of Sudoku latin squares}

\author{Jade S.~Davies and Peter J.~Dukes}
\address{
Mathematics and Statistics,
University of Victoria, Victoria, BC, Canada
}
\email{jadedavies@uvic.ca; dukes@uvic.ca}

\date{\today}
\begin{abstract}
Let $n=hw$, where $h$ and $w$ are integers with $h,w \ge 2$.
We determine the set of possible intersection numbers of two $n \times n$ latin squares having the additional `Sudoku' constraint based on a $w \times h$ grid of $h \times w$ boxes.
\end{abstract}

%\thanks{}

\maketitle
\hrule

\bigskip

\section{Introduction}
\label{sec:intro}

A \emph{latin square} of order $n$ is an $n \times n$ array with entries from a set of $n$ symbols (often taken to be $[n]:=\{1,2,\dots,n\}$) having the property that each symbol appears exactly once in every row and in every
column.  A \emph{partial latin square} of order $n$ is an $n
\times n$ array whose cells are either empty or filled with one of $n$ symbols in such a way that each symbol appears at most once in every row and in every column.
A partial latin square can be identified with the set of all ordered triples $(i,j,k)$ such that the $(i,j)$-entry contains symbol $k$.  

Given two latin squares $L$ and $L'$ of order $n$, we define their intersection $L \cap L'$ as the partial latin square consisting of all common entries in common positions of both $L$ and $L'$.  In terms of ordered triples, 
$$L \cap L' = \{(i,j,k): L_{ij}=L'_{ij}=k\}.$$

Let $\mathcal{L}_n$ denote the set of all latin squares of order $n$ with symbol set $[n]$.
The {\em intersection spectrum} for latin squares of order $n$ is defined as 
$$I(n) = \{|L \cap L'|: L,L' \in \mathcal{L}_n \}.$$ 
Notice that it is impossible for the intersection $L \cap L'$ to contain exactly one empty cell in any row or column.  It is easy to see from this that
$n^2-5,n^2-3,n^2-2,n^2-1 \not\in I(n)$. 
Define
$$\Upsilon(n):=\{0,1,2,\dots,n^2-6,n^2-4,n^2\}.$$
It was shown by Fu and Fu \cite{Fu}, and later Howell \cite{Howell}, that $I(n)=\Upsilon(n)$ for $n \ge 5$.  Below is the full determination of $I(n)$, including the small cases.

\begin{thm}
[\cite{Fu,Howell}]
\label{thm:int-ls}
$$I(n) =  
\begin{cases}
\{1\} & \mathrm{if~}n=1; \\
\{0,4\} & \mathrm{if~}n=2; \\
\{0,3,9\} & \mathrm{if~}n=3; \\
\{0,1,2,3,4,6,8,9,12,16\} & \mathrm{if~}n=4; \\
\Upsilon(n) & \mathrm{if~}n\ge 5. 
\end{cases}
$$
\end{thm}
Let $h$ and $w$ be integers with $h,w \ge 2$, and put $n=hw$. A \emph{Sudoku latin square} (or briefly \emph{Sudoku}) \emph{of type} $(h,w)$ is a latin square of order $n$ whose cells are partitioned into $h \times w$ subarrays in which each symbol appears exactly once in every subarray.  The subarrays are called \emph{boxes}, or sometimes \emph{cages}.  A partial Sudoku and completion of such is defined analogously as above for latin squares.  The problem of completing a partial Sudoku in the case $h=w=3$  is a famous recreational puzzle.  A mathematical discussion of Sudoku solving strategies can be found in \cite{Sudoku2,Sudoku1}. 

In this paper, we study the intersection problem for Sudoku latin squares of type $(h,w)$. In more detail, let $\mathcal{SL}_{h,w}$ denote the set of all Sudoku latin squares of type $(h,w)$ having a common partition into $h \times w$ boxes and a common set of  symbols. Define
$$I(h,w) = \{|L \cap L'|: L,L' \in \mathcal{SL}_{h,w} \}.$$ 
Our main result is stated below.

\begin{thm}
\label{thm:int-sudoku}
Let $h,w \ge 2$.  Then
$$I(h,w) =  
\begin{cases}
\{0,1,2,3,4,6,8,9,12,16\} & \mathrm{if~}h=w=2; \\
\Upsilon(hw) & \mathrm{otherwise}. 
\end{cases}
$$
\end{thm}
In other words, we obtain the same intersection spectrum as in Theorem~\ref{thm:int-ls}, even with the additional box/cage constraint present in Sudoku latin squares.

\section{Product constructions}

In this section, it is convenient to allow sets other than $[n]$ for row/column indexing and symbols.  Of course, a choice of row/column ordering and symbol relabeling can be imposed, and this does not affect either the latin square property nor intersection sizes.  It is also convenient to use the notation $L(x,y)$ in place of $L_{xy}$ for the entry in row $x$ and column $y$ of $L$.

As a warm-up, we review the standard Kronecker/tensor product construction for latin squares.  Suppose $L$ is a latin square of order $n$ and $M$ is a latin square of order $m$. Their \emph{product} $L \otimes M$ is defined as a latin square of order $nm$ with rows and columns indexed by $[n] \times [m]$, and entries
$$(L \otimes M)((i_1,j_1),(i_2,j_2))=(L(i_1,i_2),M(j_1,j_2)).$$
After imposing a row/column ordering and relabelling, this is equivalent to replacing the symbols in $L$ by $n$ disjoint `copies' of $M$.

\begin{ex}
A product of latin squares of orders $2$ and $3$ is shown below.  Here, we apply the re-indexing
$(1,1) \mapsto 1$, $(1,2) \mapsto 2$, $(1,3) \mapsto 3$, $(2,1) \mapsto 4$, $(2,2) \mapsto 5$, $(2,3) \mapsto 6$ for ordering rows/columns and relabeling symbols.
$$
\begin{array}{|cc|}
\hline
1&2\\
2&1\\
\hline
\end{array} \otimes
\begin{array}{|ccc|}
\hline
1&2&3\\
2&3&1\\
3&1&2\\
\hline
\end{array}
=
\begin{array}{|cccccc|}
\hline
(1,1)&(1,2)&(1,3)&(2,1)&(2,2)&(2,3)\\
(1,2)&(1,3)&(1,1)&(2,2)&(2,3)&(2,1)\\
(1,3)&(1,1)&(1,2)&(2,1)&(2,2)&(2,3)\\
(2,1)&(2,2)&(2,3)&(1,1)&(1,2)&(1,3)\\
(2,2)&(2,3)&(2,1)&(1,2)&(1,3)&(1,1)\\
(2,3)&(1,2)&(2,2)&(1,3)&(1,1)&(1,2)\\
\hline
\end{array}
\equiv
\begin{array}{|cccccc|}
\hline
1&2&3&4&5&6\\
2&3&1&5&6&4\\
3&1&2&6&4&5\\
4&5&6&1&2&3\\
5&6&4&2&3&1\\
6&4&5&3&1&2\\
\hline
\end{array}
$$
\end{ex}

A useful observation for our work to follow is that the product of latin squares of order $h$ and $w$ is, after a suitable row/column permutation, a Sudoku latin square of type $(h,w)$.

\begin{ex}
Continuing from the example above, listing the rows of the product in the order $1,4,2,5,3,6$, yields a Sudoku latin square of type $(2,3)$, with boxes as shown.
$$
\begin{array}{|cc|}
\hline
1&2\\
2&1\\
\hline
\end{array} \otimes
\begin{array}{|ccc|}
\hline
1&2&3\\
2&3&1\\
3&1&2\\
\hline
\end{array}
=
\begin{array}{|cccccc|}
\hline
1&2&3&4&5&6\\
2&3&1&5&6&4\\
3&1&2&6&4&5\\
4&5&6&1&2&3\\
5&6&4&2&3&1\\
6&4&5&3&1&2\\
\hline
\end{array}
\rightarrow
\begin{array}{|ccc|ccc|}
\hline
1&2&3&4&5&6\\
4&5&6&1&2&3\\
\hline
2&3&1&5&6&4\\
5&6&4&2&3&1\\
\hline
3&1&2&6&4&5\\
6&4&5&3&1&2\\
\hline
\end{array}
$$
\end{ex}

Let $L_1,L_2$ be latin squares of order $h$ and $M_1,M_2$ be latin squares of order $w$.  It is easy to see that $$|(L_1 \otimes M_1) \cap (L_2 \otimes M_2)| = |L_1 \cap L_2| \times |M_1 \cap M_2|.$$
Although this achieves many values in $I(h,w)$, products on their own are not sufficient to collect all intersection values.
This motivates a generalization of the product construction in which the disjoint copies of $M$ are allowed to vary.

\begin{cons}
\label{cons:triangle}
Let $L$ be a latin square of order $n$ and suppose $\mathcal{M}=\{M_k^i: i,k \in [n]\}$ is a set of $n^2$ latin squares of order $m$, where the symbol set of each $M_k^i$ is $\{(k-1)m+1,(k-1)m+2,\dots,km\}$.  Then the block matrix $L \triangleleft \mathcal{M}$
defined by 
$$(L \triangleleft \mathcal{M})((i,j),(i',j')) =  M_{L(i,i')}^i (j,j'),$$
where $i,i' \in [n]$ and $j,j' \in [m]$, is a latin square of order $mn$.  Moreover, if the rows of $L \triangleleft \mathcal{M}$ are listed in the order 
$$1,m+1,2m+1,\dots,(n-1)m+1,2,m+2,2m+2,\dots,(n-1)m+2,\dots m,2m,3m,\dots,nm,$$
then the result is a Sudoku latin square of type $(n,m)$.
\end{cons}

\begin{proof}
%[Proof of Construction~\ref{cons:triangle}]
Let $S_k$ denote the `symbol bundle' $\{(k-1)m+1, (k-1)m+2, \ldots, km\}$ for each of the squares $M_k^i$.  Symbols in $S_k$ get used in the $(i,i')$-block of $L \triangleleft \mathcal{M}$ whenever $L(i,i')=k$. The $S_k$ are disjoint intervals of integers with $\cup_{i=1}^n S_k = [mn]$.
It follows that $L \triangleleft \mathcal{M}$ is a square array of order $mn$ on symbol set $[mn]$.

Consider a row $(i,j)$ of $L \triangleleft \mathcal{M}$, where $i \in [n]$ identifies the `row bundle' and $j \in [m]$ identifies the row position within the given bundle.  This row can be partitioned into $n$ segments using the column bundles $i' \in [n]$, and each such segment contains the symbols from 
$\{M_{L(i,i')}^i(j,j'): j'\in [m]\}$. Since $L$ is a latin square, the map $i' \mapsto L(i,i')$ is a bijection, so $L(i,i')$ takes on each value $k \in [n]$ exactly once for each $i$.
        It follows that the $M_{L(i,i')}^i$ use symbol bundle $S_k$ exactly once when $i$ is fixed and $i'$ varies over $[n]$.  Moreover, each $M_{L(i,i')}^i$ is a latin square, so its $j$th row contains each symbol from $S_{L(i,i')}$ exactly once. Putting together the segments, each row $(i,j)$ of $L \triangleleft \mathcal{M}$ contains every symbol in $\cup_{k=1}^n S_k = [mn]$ exactly once.  Similar reasoning shows that each column $(i',j')$ contains every symbol exactly once.  This proves that $L \triangleleft \mathcal{M}$ is a latin square of order $mn$. 
     
    Now suppose that we have reordered the rows of 
$L \triangleleft \mathcal{M}$ as indicated in the construction. Let $(L \triangleleft \mathcal{M})^*$ denote this new array.  This is certainly still a latin square of order $mn$, so it suffices to show that $(L \triangleleft \mathcal{M})^*$ is a Sudoku of type $(n,m)$.
    
In $(L \triangleleft \mathcal{M})^*$, the first $n$ rows are those with row index $(i,1)$ in $L \triangleleft \mathcal{M}$, and in general the $j$th block of $n$ rows came from rows $(i,j)$ of $L \triangleleft \mathcal{M}$, where $i$ ranges over $[n]$. Consider the box $B_{ji'}$ of 
$(L \triangleleft \mathcal{M})^*$ defined by the $j$th block of $n$ rows and column bundle $i'$.  The $i$th row within $B_{ji'}$ has symbols from the $j$th row of $M_{L(i,i')}^i$.  Since this is a latin square on symbols $S_{L(i,i')}$, we get precisely these symbols, in some order, in this row of $B_{ji'}$. As we vary $i$,each symbol bundle $S_k$ appears exactly once in this row of the box, since $L$ is a latin square.  This shows that $B_{ji'}$ contains every symbol in $[mn]$.  It follows that the $B_{ji'}$, $j \in [m]$, $i' \in [n]$, furnish a partition of
$(L \triangleleft \mathcal{M})^*$ into an $m \times n$ array of $n \times m$ boxes, each containing every symbol exactly once, as required for a Sudoku of type $(n,m)$.
\end{proof}

We are now ready to prove our main result in all but a few cases.

\begin{proof}[Proof of Theorem~\ref{thm:int-sudoku}] (assuming $\max(h,w) \ge 5$)
Let $t \in \Upsilon(hw)$.  By interchanging the roles of $h$ and $w$ if necessary, let us assume $w \ge 5$.
We claim that $t$ can be written as a sum of $h^2$ integers in $\Upsilon(w)$. 

Using the division algorithm, write $t=qw^2+r$, where $0 \le r \le w^2-1$.  If $q=h^2$, then $r=0$ and we have $t = h^2 w^2$.  If $q=h^2-1$, we have $r \in \Upsilon(w)$ and can take $t = (h^2-1)w^2+r$.  Suppose $q \le h^2-2$.  If $r \le w^2-6$, then $r \in \Upsilon(w)$ and we're done as above.  Otherwise, we can write $r=(w^2-6)+s$, where $s \in \{1,2,3,4,5\}$, and we have $r,s \in \Upsilon(w)$.

Suppose then that $t=\sum_{i,j=1}^h t_{ij}$, where $t_{ij} \in \Upsilon(w)$.
Let $L$ be any latin square of order $h$.  Let $\mathcal{M}=\{M_i^j:i,j \in [h]\}$ and $\mathcal{N}=\{N_i^j:i,j \in [h]\}$ be families of latin squares of order $w$, with symbol sets as in Construction~\ref{cons:triangle}, and such that 
$|M_i^j \cap N_i^j|=t_{ij}$.
Then 
$$t=\sum_{i,j=1}^h t_{ij} 
=|(L \triangleleft \mathcal{M}) 
\cap (L \triangleleft \mathcal{N})| 
=|(L \triangleleft \mathcal{M})^* 
\cap (L \triangleleft \mathcal{N})^*| 
\in I(h,w),$$ as required.
\end{proof}

\section{Small cases and special constructions}

In this section, we consider $I(h,w)$ for the remaining cases in which $\max(h,w) \le 4$.  
As before, we may assume without loss of generality that $w \ge h$.  When $w=4$, we note in the following lemma that $I(4)$ can be used to realize most intersection values in $\Upsilon(n)$ for $I(h,4)$.

\begin{lemma}
\label{lem:w=4}
Let $h\in \{2,3,4\}$, and put $n=4h$. Then 
$$I(h,4) \supseteq \Upsilon(n) \setminus \{n^2-x:x =6,9,11\}.$$
\end{lemma}

\begin{proof}
Let $t\in\Upsilon(n)\setminus\{n^2-x:x=6,9,11\}$. Similar to the proof above, we use intersections of the form $(L \triangleleft \mathcal{M}) \cap (L \triangleleft \mathcal{N})$ where $L$ is a latin square of order $h$ and those in $\mathcal{M},\mathcal{N}$ have order $4$. It suffices to show that $t$ can be written as a sum of $h^2$ integers from $\Upsilon(4)$. Write $t=16q+r$, where $0\leq r\leq 15$. If $q=h^2$ then $r=0$ and we simply use pairs of identical latin squares of order $4$ to obtain the intersection value $t=16h^2 \in \Upsilon(n)$. If $q=h^2-1$, then the restriction on $t$ implies $r\in \Upsilon(4)$.  We use $q=h^2-1$ summands $16$ together with $r$.

Now suppose $q \le h^2-2$. We claim that each value $r\in\{0,1,\ldots,16\}$ can be expressed as a sum of two integers $k,l\in\Upsilon(4)$. If $r\in\Upsilon(4)$ simply take $k=r$ and $l=0$. If $r\not\in\Upsilon(4)$ we use Table~\ref{tab:ups4} to find suitable values for $k$ and $l$.
\begin{center}
\begin{table}[htbp]
    \begin{tabular}{c|cccccccc}
  $r$ & 5 & 7 & 10 & 11 & 13 & 14 & 15 \\
  \cline{1-8}
$k$ & 3 & 4 & 6 & 8 & 9 & 8 & 9 & $\in \Upsilon(4)$\\
$l$ & 2 & 3 & 4 & 3 & 4 & 6 & 6 & $\in \Upsilon(4)$\\
\end{tabular}
    \caption{Expressing each $r \not\in \Upsilon(4)$ as a sum $r=k+l$, where $k,l\in\Upsilon(4)$.}
    \label{tab:ups4}
\end{table}
\end{center}

It follows that for $q\leq h^2-2$, we can express $t=16q+k+l$ and obtain all intersection values $t<16(h^2-1)$. Together with the above, this shows that $I(h,4) \supseteq \Upsilon(n) \setminus \{n^2-x:x =6,9,11\}$.
\end{proof}

We implemented a computer search to achieve the missing values in Lemma~\ref{lem:w=4} for each $h=2,3,4$, and also to show $I(h,w)=I(hw)$ when $h,w \le 3$.  As it turns out, we were able in all cases to identify a single Sudoku latin square $L$, and family of Sudoku latin squares $\{L_i\}$ of the same type such that $|L \cap L_i|=i$.  The basic method simply involved generating several Sudoku latin squares, and collecting those with a variety of intersection values.  However, it is difficult to achieve intersection values $i$ near the maximum $n^2$ `at random'.  To speed up these cases, we implemented a variation of the latin square Markov process \cite{JM} to drift only a few steps from $L$.  This usually produces a Sudoku latin square $L$ with relatively large 
$|L \cap L'|$.  Likewise, to collect small intersection values, it was helpful to start the Markov process on a row-derangement of $L$.  After a few steps, the resulting Sudoku latin square $L'$ usually has small intersection with $L$.

\begin{lemma}
\label{lem:exceptions4}
For $h\in \{2,3,4\}$ and $n=4h$, $I(h,4) \supseteq \{n^2-x:x =6,9,11\}.$
\end{lemma}

\begin{proof}
We give latin squares $L$ and $L_i$, found as in the discussion above, for the three exceptional values of $i$.  Discrepancies in $L_i$ from entries in $L$ are displayed in bold.
We handle the case $h=2$ below.  The cases $h=3$ and $h=4$ are done similarly in the Appendix.
\end{proof}

\begin{center}
$\begin{array}{rr}
L=\begin{array}{|cccc|cccc|}
\hline
0 & 1 & 2 & 3 & 4 & 5 & 6 & 7 \\
6 & 4 & 7 & 5 & 2 & 1 & 3 & 0 \\
\hline
5 & 2 & 4 & 6 & 3 & 0 & 7 & 1 \\
1 & 3 & 0 & 7 & 6 & 4 & 2 & 5 \\
\hline
2 & 5 & 6 & 1 & 0 & 7 & 4 & 3 \\
7 & 0 & 3 & 4 & 5 & 6 & 1 & 2 \\
\hline
3 & 6 & 1 & 0 & 7 & 2 & 5 & 4 \\
4 & 7 & 5 & 2 & 1 & 3 & 0 & 6 \\
\hline
\end{array}
&
\hspace{1cm}
L_{53}=\begin{array}{|cccc|cccc|}
\hline
0 & \bf 2 & \bf 1 & 3 & 4 & 5 & 6 & 7 \\
6 & 4 & 7 & 5 & 2 & 1 & 3 & 0 \\
\hline
5 & \bf 6 & 4 & \bf 2 & 3 & 0 & 7 & 1 \\
1 & 3 & 0 & 7 & 6 & 4 & 2 & 5 \\
\hline
2 & 5 & 6 & 1 & 0 & 7 & 4 & 3 \\
7 & 0 & 3 & 4 & 5 & \bf 2 & 1 & \bf 6 \\
\hline
3 & \bf 1 & \bf 2 & 0 & 7 & \bf 6 & 5 & 4 \\
4 & 7 & 5 & \bf 6 & 1 & 3 & 0 & \bf 2 \\
\hline
\end{array}\\
\\
L_{55}=\begin{array}{|cccc|cccc|}
\hline
0 & 1 & \bf 3 & \bf 2 & 4 & 5 & 6 & 7 \\
6 & 4 & 7 & 5 & 2 & 1 & 3 & 0 \\
\hline
\bf 4 & 2 & \bf 5 & 6 & 3 & 0 & 7 & 1 \\
1 & 3 & 0 & 7 & 6 & 4 & 2 & 5 \\
\hline
2 & 5 & 6 & 1 & 0 & 7 & 4 & 3 \\
7 & 0 & \bf 4 & \bf 3 & 5 & 6 & 1 & 2 \\
\hline
3 & 6 & 1 & 0 & 7 & 2 & 5 & 4 \\
\bf 5 & 7 & \bf 2 & \bf 4 & 1 & 3 & 0 & 6 \\
\hline
\end{array}
&
\hspace{1cm}
L_{58}=\begin{array}{|cccc|cccc|}
\hline
0 & 1 & 2 & 3 & 4 & 5 & 6 & 7 \\
\bf 7 & 4 & \bf 6 & 5 & 2 & 1 & 3 & 0 \\
\hline
5 & 2 & 4 & 6 & 3 & 0 & 7 & 1 \\
1 & 3 & 0 & 7 & 6 & 4 & 2 & 5 \\
\hline
2 & 5 & \bf 7 & 1 & 0 & \bf 6 & 4 & 3 \\
\bf 6 & 0 & 3 & 4 & 5 & \bf 7 & 1 & 2 \\
\hline
3 & 6 & 1 & 0 & 7 & 2 & 5 & 4 \\
4 & 7 & 5 & 2 & 1 & 3 & 0 & 6 \\
\hline
\end{array}
\end{array}$
\end{center}

Only the cases with $h \le w \le 3$ remain.  For $(h,w)=(2,2)$, we again display a Sudoku latin square $L$, followed by a list of Sudoku latin squares $L_i$, each of which intersects $L$ in exactly $i$ positions.  This establishes that $I(2,2)=I(4)$ as stated in Theorem~\ref{thm:int-ls}.

$L=\begin{array}{|cc|cc|}
\hline
0&1&2&3\\ 2&3&0&1\\ \hline 3&0&1&2\\ 1&2&3&0\\
\hline
\end{array}$

\begin{center}
$
\begin{array}{ccc}
L_0=\begin{array}{|cc|cc|}
\hline
3&2&0&1\\ 0&1&3&2\\ \hline 1&3&2&0\\ 2&0&1&3\\
\hline
\end{array}
&
L_1=\begin{array}{|cc|cc|}
\hline
3&2&1&0\\ 0&1&3&2\\ \hline 2&3&0&1\\ 1&0&2&3\\
\hline
\end{array}
&
L_2=\begin{array}{|cc|cc|}
\hline
2&3&0&1\\ 1&0&3&2\\ \hline 3&1&2&0\\ 0&2&1&3\\
\hline
\end{array}\\
\\
L_3=\begin{array}{|cc|cc|}
\hline
2&3&0&1\\ 1&0&2&3\\ \hline 3&2&1&0\\ 0&1&3&2\\
\hline
\end{array}
&
L_4=\begin{array}{|cc|cc|}
\hline
1&3&0&2\\ 2&0&3&1\\ \hline 3&2&1&0\\ 0&1&2&3\\
\hline
\end{array}
&
L_6=\begin{array}{|cc|cc|}
\hline
2&0&1&3\\ 1&3&0&2\\ \hline 3&1&2&0\\ 0&2&3&1\\
\hline
\end{array}\\
\\
L_8=\begin{array}{|cc|cc|}
\hline
0&3&2&1\\ 2&1&0&3\\ \hline 3&2&1&0\\ 1&0&3&2\\
\hline
\end{array}
&
L_9=\begin{array}{|cc|cc|}
\hline
1&0&2&3\\ 2&3&0&1\\ \hline 0&1&3&2\\ 3&2&1&0\\
\hline
\end{array}
&
L_{12}=\begin{array}{|cc|cc|}
\hline
2&1&0&3\\ 0&3&2&1\\ \hline 3&0&1&2\\ 1&2&3&0\\
\hline
\end{array}
\end{array}$
\end{center}
The larger cases $(h,w)=(2,3)$ and $(3,3)$ are addressed in the Appendix.  These, together with Lemmas~\ref{lem:w=4} and \ref{lem:exceptions4} complete the proof of Theorem~\ref{thm:int-sudoku}.

\section{Conclusion and Discussion}

We have shown that Sudoku latin squares with $h \times w$ rectangular cages achieve the same intersection values as (unrestricted) latin squares of order $n=hw$.  This is  unsurprising for large $n$, but perhaps noteworthy that there are no exceptions for smaller sizes.

To conclude, we offer some thoughts on possible next directions for the research.  First, it is natural to be curious about non-rectangular cages, especially in the first interesting case $n=5$.  A `Pentadoku' is a $5 \times 5$ latin square with five pentomino cages (in some arrangement).  As in Sudoku, each cage of a Pentadoku must contain each symbol exactly once. Figure~\ref{fig:pentadoku} shows an example.

\begin{figure}[htbp]
%\begin{minipage}[c]{0.45\linewidth}
\begin{center}
\begin{tikzpicture}[scale=0.8]
\foreach \a in {0,1,...,5}
                \draw (\a,0)--(\a,5);
\foreach \a in {0,1,...,5}
                \draw (0,\a)--(5,\a);
\foreach \a in {0,5}
                \draw[line width=2pt] (\a,0)--(\a,5);
\foreach \a in {0,5}
                \draw[line width=2pt] (0,\a)--(5,\a);
\draw[line width=2pt] (0,4)--(1,4)--(1,3)--(4,3);
\draw[line width=2pt] (2,5)--(2,4)--(4,4)--(4,2)--(5,2);
\draw[line width=2pt] (0,1)--(2,1)--(2,3);
\draw[line width=2pt] (4,0)--(4,1)--(3,1)--(3,2)--(2,2);
%%%
\node at (0.5,4.5) {\small 1};
\node at (1.5,4.5) {\small 2};
\node at (2.5,4.5) {\small 3};
\node at (3.5,4.5) {\small 4};
\node at (4.5,4.5) {\small 5};
%%%
\node at (0.5,3.5) {\small 2};
\node at (1.5,3.5) {\small 4};
\node at (2.5,3.5) {\small 5};
\node at (3.5,3.5) {\small 3};
\node at (4.5,3.5) {\small 1};
%%%
\node at (0.5,2.5) {\small 4};
\node at (1.5,2.5) {\small 3};
\node at (2.5,2.5) {\small 1};
\node at (3.5,2.5) {\small 5};
\node at (4.5,2.5) {\small 2};
%%%
\node at (0.5,1.5) {\small 5};
\node at (1.5,1.5) {\small 1};
\node at (2.5,1.5) {\small 4};
\node at (3.5,1.5) {\small 2};
\node at (4.5,1.5) {\small 3};
%%%
\node at (0.5,0.5) {\small 3};
\node at (1.5,0.5) {\small 5};
\node at (2.5,0.5) {\small 2};
\node at (3.5,0.5) {\small 1};
\node at (4.5,0.5) {\small 4};
\end{tikzpicture} 
\end{center}
\caption{A completed Pentadoku puzzle}
\label{fig:pentadoku}
\end{figure}

Here, we  summarize the outcome of a short computer search on Pentadoku intersections.  Up to symmetries, there are 107 different tilings of a $5 \times 5$ grid using distinct pentominos.  Four of these tilings admit no latin square which satisfies the corresponding cage condition.  Of the remaining 103 tilings, 58 have `full' intersection spectrum equal to $\Upsilon(5)$.  A further 44 tilings achieve `most of' $\Upsilon(5)$, with the exception of some subset of the values $\{1,14,16,17,18\}$.  Finally, one tiling (the one displayed in Figure~\ref{fig:pentadoku}) has a unique latin square up to symbol relabelling.  As a result, this tiling only achieves intersections in 
$\{0,5,10,15,25\}$ (via symbol permutations).
In future work, it may be interesting to consider Sudoku intersections for larger sizes in the presence of non-rectangular cage patterns.  Perhaps some results are possible when cages are `nearly' rectangular.

There are a few conceivable ways to generalize the `dimension' of the Sudoku intersection problem.  First, some preliminary work has been done \cite{Fu-cubes} on intersections of latin cubes and hypercubes.  A Sudoku-type constraint could be imposed and investigated.  However, a full solution of the latin intersection problem, even in dimension three, should probably precede any attempt to impose additional Sudoku conditions.  Alternatively, one could consider traditional latin squares which have more than one simultaneous box/cage partition, say into $h_1 \times w_1$ and $h_2 \times w_2$ rectangles where $n=h_1w_1=h_2w_2$.  Another direction could involve intersections achieved between three or more Sudoku latin squares.  For any of these extensions,  Construction~\ref{cons:triangle}, or some minor extension on this theme, may be useful as a starting point.

Finally, in \cite{DH}, a two-parameter version of latin square intersections was studied.  Let $L$ be a latin square of order
$n$ and $M$ a latin square of order $m \ge n$.  Assume the sets of rows, columns and symbols of $M$ contain the corresponding sets for $L$.  In this way, $L$ and $M$ are partial latin squares on the same rows, columns, and symbols.  Define $L \cap M$ to be the corresponding intersection of partial latin squares.  For fixed integers $m \ge n \ge 1$, the set of all possible values of $|L \cap M|$ was determined in \cite{DH}.  A common extension of this work and ours on Sudoku might be a reasonable problem for future consideration.

\section*{Appendix: $I(2,3)$, $I(3,3)$, $I(3,4)$, $I(4,4)$}

We give data for the remaining cases of pairs $(h,w)$.
For the first two cases, Sudoku latin squares are written in single-line notation.  The integer preceding each square is the intersection number when compared with the last square in the list.  This verifies $I(2,3)=\Upsilon(6)$ and $I(3,3)=\Upsilon(9)$.

$(h,w)=(2,3)$\\
0: \small{\tt 541230|230514|053421|412053|305142|124305}\\
1: \small{\tt 524130|031524|415302|203451|340215|152043}\\
2: \small{\tt 523401|014253|452130|130542|205314|341025}\\
3: \small{\tt 045132|231450|453021|120543|304215|512304}\\
4: \small{\tt 045213|321045|152304|430152|204531|513420}\\
5: \small{\tt 214035|053412|421350|530241|142503|305124}\\
6: \small{\tt 132054|450132|521340|043521|304215|215403}\\
7: \small{\tt 523401|410532|201354|354210|032145|145023}\\
8: \small{\tt 250134|143502|502341|431250|324015|015423}\\
9: \small{\tt 412503|305142|234015|150234|541320|023451}\\
10: \small{\tt 421305|035412|154230|302541|213054|540123}\\
11: \small{\tt 253041|041253|105432|324105|432510|510324}\\
12: \small{\tt 135042|240513|014235|523104|451320|302451}\\
13: \small{\tt 503412|124305|012534|345120|251043|430251}\\
14: \small{\tt 412503|503142|325410|140235|051324|234051}\\
15: \small{\tt 103542|452103|015234|234015|541320|320451}\\
16: \small{\tt 104235|253104|015342|432510|541023|320451}\\
17: \small{\tt 210345|345201|102534|534120|421053|053412}\\
18: \small{\tt 201345|543201|015432|432510|154023|320154}\\
19: \small{\tt 023145|541032|205314|314520|152403|430251}\\
20: \small{\tt 543102|012345|305214|124530|451023|230451}\\
21: \small{\tt 210453|543102|105234|324015|431520|052341}\\
22: \small{\tt 103542|542130|315204|024315|451023|230451}\\
23: \small{\tt 012345|453201|125430|304512|541023|230154}\\
24: \small{\tt 012453|543102|305214|124530|431025|250341}\\
25: \small{\tt 012345|453102|501234|234510|145023|320451}\\
26: \small{\tt 021345|534102|105234|342510|453021|210453}\\
27: \small{\tt 012543|543102|105234|324051|451320|230415}\\
28: \small{\tt 012345|543012|105234|324501|450123|231450}\\
29: \small{\tt 210345|543102|105234|324510|051423|432051}\\
30: \small{\tt 012345|543201|105432|324510|451023|230154}\\
32: \small{\tt 102345|543102|015234|324510|451023|230451}\\
36: \small{\tt 012345|543102|105234|324510|451023|230451}\\

\vspace{1cm}
$(h,w)=(3,3)$\\
0: \small{\tt 364817205|812605473|705234816|576082134|083146527|421573680|238460751|140758362|657321048}\\
1: \small{\tt 503681724|867204315|421753806|145826073|286037451|370415682|758362140|032148567|614570238}\\
2: \small{\tt 143782065|780465321|265103784|524670813|671834502|308521476|017346258|836257140|452018637}\\
3: \small{\tt 605342781|437081526|281765403|823407165|146853072|570126348|764538210|058214637|312670854}\\
4: \small{\tt 026518347|753246801|148370562|602453718|431782650|587601234|315824076|874065123|260137485}\\
5: \small{\tt 536820147|821347605|047156823|402781356|785463210|613205784|168034572|254678031|370512468}\\
6: \small{\tt 086235714|237814605|514706238|408127356|125368470|673450821|752681043|361042587|840573162}\\
7: \small{\tt 647230158|231458076|850167234|084521367|526374810|713806425|102645783|375082641|468713502}\\
8: \small{\tt 370162548|651784023|248530167|827415630|413607852|065823714|534071286|106248375|782356401}\\
9: \small{\tt 526380714|387214605|014756382|402137856|135862470|678405231|253671048|861043527|740528163}\\
10: \small{\tt 874523601|526701438|301684725|137056284|058247163|462138057|685470312|240315876|713862540}\\
11: \small{\tt 648705312|703412856|521368704|254173068|176084235|830256471|367821540|012547683|485630127}\\
12: \small{\tt 821365704|367204518|405781362|142037685|038652147|576148230|683520471|750413826|214876053}\\
13: \small{\tt 342710586|571684203|086235147|607158432|853427610|214306758|135872064|428061375|760543821}\\
14: \small{\tt 862105374|103647258|547382106|756813042|418026735|230754681|371568420|084271563|625430817}\\
15: \small{\tt 421375086|370286145|856041372|682430751|735812604|104657238|043128567|518763420|267504813}\\
16: \small{\tt 064821375|823675410|175304826|516783204|780246531|432510687|301462758|247158063|658037142}\\
17: \small{\tt 038527614|476310852|512648073|253106748|104783265|867452301|740831526|681275430|325064187}\\
18: \small{\tt 816542037|540137682|237086541|781350426|352461708|604728153|025613874|463875210|178204365}\\
19: \small{\tt 052386417|384517260|617402358|760153824|135824706|248760531|406238175|821675043|573041682}\\
20: \small{\tt 617048325|043125786|528367041|864203517|250471638|731586402|306712854|472850163|185634270}\\
21: \small{\tt 148523670|526470381|370618524|034756812|751284063|862031457|613847205|287305146|405162738}\\
22: \small{\tt 378645021|640721358|521038647|854260173|263417805|107853462|036172584|412586730|785304216}\\
23: \small{\tt 046123578|125478630|378506124|830714256|714265803|652830417|201687345|567341082|483052761}\\
24: \small{\tt 031426578|425378160|876501423|683754201|740213856|152860347|504137682|217685034|368042715}\\
25: \small{\tt 012463758|437158260|658720413|863547102|540312876|271806345|704235681|125684037|386071524}\\
26: \small{\tt 061235748|237684150|548701236|856427301|420316875|173850624|704162583|682543017|315078462}\\
27: \small{\tt 158243670|246570831|370618245|035726418|721485063|864031527|612857304|487302156|503164782}\\
28: \small{\tt 528341670|146270583|370658142|032716458|715482036|864035217|651807324|407123865|283564701}\\
29: \small{\tt 813475620|476120358|520683471|051246783|742831065|368057142|684312507|237504816|105768234}\\
30: \small{\tt 415372608|376508124|208641375|821736540|034815762|567024831|643150287|750283416|182467053}\\
31: \small{\tt 024835671|836271450|571604832|152786304|780342165|463150287|605428713|348517026|217063548}\\
32: \small{\tt 815370642|342651708|076482351|201536487|538147260|764208135|683715024|457023816|120864573}\\
33: \small{\tt 062371548|135684270|784502316|416735802|370826451|258410637|503267184|827143065|641058723}\\
34: \small{\tt 421035678|076248135|538671042|842706351|750312864|163854207|604127583|317580426|285463710}\\
35: \small{\tt 412835670|836170254|570642831|054781362|781326045|263054187|645218703|328407516|107563428}\\
36: \small{\tt 721045638|046238157|538671042|852306471|307412865|164857203|670123584|413580726|285764310}\\
37: \small{\tt 017645238|648132750|532708641|870351462|361427805|254860173|786013524|423586017|105274386}\\
38: \small{\tt 214385607|386107452|507624381|051836724|732041865|468752130|623410578|840573216|175268043}\\
39: \small{\tt 072543618|351678240|864102537|740836125|538421706|216750483|603217854|427385061|185064372}\\
40: \small{\tt 152360478|346578201|078412365|805734612|731625840|264801537|413257086|627083154|580146723}\\
41: \small{\tt 712845630|846130257|530672841|054386172|387421065|261057483|678213504|423508716|105764328}\\
42: \small{\tt 572103648|346875210|018642375|850231467|731456082|264780531|623517804|407328156|185064723}\\
43: \small{\tt 512048673|046173285|873652041|381706452|765421308|204385167|650217834|427830516|138564720}\\
44: \small{\tt 021345678|634178520|578602341|852731064|710426835|463850217|287013456|346587102|105264783}\\
45: \small{\tt 072543618|546718230|318602547|837156402|150427863|264830751|605271384|421385076|783064125}\\
46: \small{\tt 012345687|376182540|845607321|584763102|730421865|261850734|103276458|427538016|658014273}\\
47: \small{\tt 012735648|736148250|548602731|851427306|470361825|263850174|607214583|324586017|185073462}\\
48: \small{\tt 072315486|346078251|518642370|864753102|730821645|251460837|103287564|427536018|685104723}\\
49: \small{\tt 047132658|316548270|528607341|835716402|470823165|261450837|603271584|752384016|184065723}\\
50: \small{\tt 028534671|346178250|571602384|853761402|710423865|264850137|632017548|405286713|187345026}\\
51: \small{\tt 013845627|426173058|578602341|854736102|730421865|261058473|102367584|347580216|685214730}\\
52: \small{\tt 042615378|316784250|578230641|861507432|730421865|254863107|603172584|427358016|185046723}\\
53: \small{\tt 012354678|364178250|578602134|851736402|740821365|236045817|407213586|623587041|185460723}\\
54: \small{\tt 012345678|546178302|378620451|851736240|260451837|734802165|103267584|427583016|685014723}\\
55: \small{\tt 512873640|346105278|078642351|831726405|760451832|254380167|603217584|427538016|185064723}\\
56: \small{\tt 012345678|346871205|578602341|451736820|837120564|260458137|605217483|724583016|183064752}\\
57: \small{\tt 012345678|346817052|578602341|851736204|730428165|264150837|407283516|623571480|185064723}\\
58: \small{\tt 012483675|364175208|578602341|831726450|740531862|256840137|603217584|427358016|185064723}\\
59: \small{\tt 012345678|346817250|578602341|751436802|830251467|264780135|603178524|487523016|125064783}\\
60: \small{\tt 210357648|376148250|548602371|854736102|732410865|061825437|603271584|427583016|185064723}\\
61: \small{\tt 512740638|746138250|038652741|851376402|370421865|264805173|603217584|427583016|185064327}\\
62: \small{\tt 012347658|346185270|578602341|851736402|730421865|624850137|263578014|407213586|185064723}\\
63: \small{\tt 012348675|346175280|578602341|854736012|730421856|261850437|403217568|627583104|185064723}\\
64: \small{\tt 012345678|546078231|378162450|854736102|730421865|261850347|603217584|427583016|185604723}\\
65: \small{\tt 012345768|346178250|578602341|851736402|730421586|264850137|603517824|487263015|125084673}\\
66: \small{\tt 082345671|346178520|517602348|851736402|730421865|264850137|603517284|475283016|128064753}\\
67: \small{\tt 028345671|346178250|517602348|251836407|730421865|864750132|603217584|472583016|185064723}\\
68: \small{\tt 012345678|346178205|578602341|831726450|760451832|254830167|603217584|427583016|185064723}\\
69: \small{\tt 012435678|346178250|578602341|857346102|430721865|261850437|603217584|724583016|185064723}\\
70: \small{\tt 012345678|346178250|578062341|854736102|730821465|261450837|603217584|427583016|185604723}\\
71: \small{\tt 012345678|346178250|578602341|251736804|830421567|764850132|603217485|427583016|185064723}\\
72: \small{\tt 012345678|346178205|578602341|831726450|750431862|264850137|603217584|427583016|185064723}\\
73: \small{\tt 082345671|346178250|571602348|158736402|730421865|264850137|603217584|427583016|815064723}\\
74: \small{\tt 012345678|346178250|578602341|851736402|730421865|624850137|403217586|267583014|185064723}\\
75: \small{\tt 012345678|346178250|578602341|185736402|730421865|264850137|603217584|427583016|851064723}\\
77: \small{\tt 012345678|346178250|578602341|851736402|703421865|264850137|630217584|427583016|185064723}\\
81: \small{\tt 012345678|346178250|578602341|851736402|730421865|264850137|603217584|427583016|185064723}
\normalsize

\vspace{2cm}
For cases with $w=4$, it suffices to give the three exceptional cases not covered by Lemma~\ref{lem:w=4}.  For each type, we give a single Sudoku latin square $L$ and three additional squares $L_i$ which intersect $L$ in exactly $i$ positions, where $i$ ranges over $\{n^2-x: x=6,9,11\}$.

\newpage
$(h,w)=(3,4)$
$$\hspace{4mm}
L=\small\begin{array}{|cccc|cccc|cccc|}
\hline
0 & 1 & 2 & 3 & 4 & 5 & 6 & 7 & 8 & 9 & 10 & 11 \\
7 & 4 & 8 & 5 & 9 & 1 & 11 & 10 & 3 & 6 & 2 & 0 \\
11 & 10 & 9 & 6 & 2 & 8 & 3 & 0 & 4 & 5 & 1 & 7 \\
\hline
6 & 8 & 5 & 1 & 3 & 0 & 4 & 2 & 9 & 11 & 7 & 10 \\
10 & 3 & 0 & 11 & 7 & 6 & 5 & 9 & 1 & 4 & 8 & 2 \\
9 & 2 & 7 & 4 & 10 & 11 & 1 & 8 & 5 & 0 & 6 & 3 \\
\hline
5 & 6 & 3 & 9 & 0 & 10 & 8 & 4 & 7 & 2 & 11 & 1 \\
1 & 7 & 11 & 2 & 6 & 3 & 9 & 5 & 0 & 10 & 4 & 8 \\
4 & 0 & 10 & 8 & 1 & 2 & 7 & 11 & 6 & 3 & 5 & 9 \\
\hline
8 & 11 & 1 & 0 & 5 & 9 & 10 & 6 & 2 & 7 & 3 & 4 \\
2 & 5 & 4 & 10 & 8 & 7 & 0 & 3 & 11 & 1 & 9 & 6 \\
3 & 9 & 6 & 7 & 11 & 4 & 2 & 1 & 10 & 8 & 0 & 5 \\
\hline
\end{array}$$

\normalsize
$$L_{133}=\small\begin{array}{|cccc|cccc|cccc|}
\hline
0 & \textbf{2} & \textbf{1} & 3 & 4 & 5 & 6 & 7 & 8 & 9 & 10 & 11 \\
7 & 4 & 8 & 5 & 9 & 1 & 11 & 10 & 3 & 6 & 2 & 0 \\
11 & 10 & 9 & 6 & 2 & 8 & 3 & 0 & 4 & 5 & 1 & 7 \\
\hline
6 & 8 & 5 & 1 & 3 & 0 & 4 & 2 & 9 & 11 & 7 & 10 \\
10 & 3 & \textbf{11} & \textbf{0} & 7 & 6 & 5 & 9 & 1 & 4 & 8 & 2 \\
9 & \textbf{7} & \textbf{2} & 4 & 10 & 11 & 1 & 8 & 5 & 0 & 6 & 3 \\
\hline
5 & 6 & 3 & 9 & 0 & 10 & 8 & 4 & 7 & 2 & 11 & 1 \\
1 & \textbf{11} & \textbf{7} & 2 & 6 & 3 & 9 & 5 & 0 & 10 & 4 & 8 \\
4 & 0 & 10 & 8 & 1 & 2 & 7 & 11 & 6 & 3 & 5 & 9 \\
\hline
8 & \textbf{1} & \textbf{0} & \textbf{11} & 5 & 9 & 10 & 6 & 2 & 7 & 3 & 4 \\
2 & 5 & 4 & 10 & 8 & 7 & 0 & 3 & 11 & 1 & 9 & 6 \\
3 & 9 & 6 & 7 & 11 & 4 & 2 & 1 & 10 & 8 & 0 & 5 \\
\hline
\end{array}$$

\normalsize
$$L_{135}=\small\begin{array}{|cccc|cccc|cccc|}
\hline
\textbf{11} & 1 & 2 & 3 & 4 & 5 & 6 & 7 & 8 & 9 & 10 & \textbf{0} \\
\textbf{0} & 4 & 8 & 5 & 9 & 1 & 11 & 10 & 3 & 6 & 2 & \textbf{7} \\
\textbf{7} & 10 & 9 & 6 & 2 & 8 & 3 & 0 & 4 & 5 & \textbf{11} & \textbf{1} \\
\hline
6 & 8 & 5 & 1 & 3 & 0 & 4 & 2 & 9 & 11 & 7 & 10 \\
10 & 3 & 0 & 11 & 7 & 6 & 5 & 9 & 1 & 4 & 8 & 2 \\
9 & 2 & 7 & 4 & 10 & 11 & 1 & 8 & 5 & 0 & 6 & 3 \\
\hline
5 & 6 & 3 & 9 & 0 & 10 & 8 & 4 & 7 & 2 & \textbf{1} & \textbf{11} \\
1 & 7 & 11 & 2 & 6 & 3 & 9 & 5 & 0 & 10 & 4 & 8 \\
4 & 0 & 10 & 8 & 1 & 2 & 7 & 11 & 6 & 3 & 5 & 9 \\
\hline
8 & 11 & 1 & 0 & 5 & 9 & 10 & 6 & 2 & 7 & 3 & 4 \\
2 & 5 & 4 & 10 & 8 & 7 & 0 & 3 & 11 & 1 & 9 & 6 \\
3 & 9 & 6 & 7 & 11 & 4 & 2 & 1 & 10 & 8 & 0 & 5 \\
\hline
\end{array}$$

\normalsize
$$L_{138}=\small\begin{array}{|cccc|cccc|cccc|}
\hline
0 & 1 & 2 & 3 & 4 & 5 & 6 & 7 & 8 & 9 & 10 & 11 \\
7 & 4 & 8 & 5 & 9 & 1 & 11 & 10 & 3 & 6 & 2 & 0 \\
11 & 10 & 9 & 6 & 2 & 8 & 3 & 0 & 4 & 5 & 1 & 7 \\
\hline
6 & 8 & 5 & 1 & 3 & 0 & 4 & 2 & 9 & 11 & 7 & 10 \\
10 & 3 & 0 & 11 & 7 & 6 & 5 & 9 & 1 & 4 & 8 & 2 \\
9 & 2 & 7 & 4 & 10 & 11 & 1 & 8 & 5 & 0 & 6 & 3 \\
\hline
5 & \textbf{7} & 3 & 9 & \textbf{6} & 10 & 8 & 4 & \textbf{0} & 2 & 11 & 1 \\
1 & \textbf{6} & 11 & 2 & \textbf{0} & 3 & 9 & 5 & \textbf{7} & 10 & 4 & 8 \\
4 & 0 & 10 & 8 & 1 & 2 & 7 & 11 & 6 & 3 & 5 & 9 \\
\hline
8 & 11 & 1 & 0 & 5 & 9 & 10 & 6 & 2 & 7 & 3 & 4 \\
2 & 5 & 4 & 10 & 8 & 7 & 0 & 3 & 11 & 1 & 9 & 6 \\
3 & 9 & 6 & 7 & 11 & 4 & 2 & 1 & 10 & 8 & 0 & 5 \\
\hline
\end{array}$$
\normalsize

\newpage
$(h,w)=(4,4)$

$$\hspace{3mm} L=\tiny \begin{array}{|cccc|cccc|cccc|cccc|}
\hline
0 & 1 & 2 & 3 & 4 & 5 & 6 & 7 & 8 & 9 & 10 & 11 & 12 & 13 & 14 & 15 \\
6 & 10 & 9 & 5 & 13 & 11 & 14 & 12 & 2 & 7 & 4 & 15 & 8 & 0 & 1 & 3 \\
15 & 7 & 14 & 4 & 8 & 10 & 1 & 3 & 5 & 0 & 13 & 12 & 2 & 9 & 11 & 6 \\
8 & 12 & 13 & 11 & 2 & 0 & 15 & 9 & 1 & 14 & 6 & 3 & 4 & 7 & 5 & 10 \\
\hline
3 & 0 & 11 & 6 & 14 & 13 & 2 & 4 & 9 & 10 & 1 & 8 & 7 & 5 & 15 & 12 \\
10 & 5 & 4 & 8 & 12 & 6 & 7 & 15 & 14 & 13 & 2 & 0 & 3 & 11 & 9 & 1 \\
14 & 9 & 12 & 1 & 10 & 3 & 8 & 5 & 15 & 4 & 11 & 7 & 6 & 2 & 0 & 13 \\
2 & 13 & 7 & 15 & 1 & 9 & 11 & 0 & 12 & 6 & 3 & 5 & 10 & 14 & 8 & 4 \\
\hline
5 & 6 & 8 & 9 & 7 & 12 & 13 & 2 & 0 & 11 & 15 & 10 & 1 & 3 & 4 & 14 \\
13 & 14 & 3 & 2 & 6 & 1 & 0 & 11 & 4 & 5 & 8 & 9 & 15 & 12 & 10 & 7 \\
11 & 4 & 0 & 12 & 15 & 8 & 3 & 10 & 13 & 1 & 7 & 14 & 5 & 6 & 2 & 9 \\
1 & 15 & 10 & 7 & 9 & 4 & 5 & 14 & 3 & 2 & 12 & 6 & 11 & 8 & 13 & 0 \\
\hline
12 & 2 & 5 & 14 & 0 & 15 & 10 & 8 & 7 & 3 & 9 & 1 & 13 & 4 & 6 & 11 \\
9 & 8 & 15 & 10 & 11 & 7 & 4 & 13 & 6 & 12 & 0 & 2 & 14 & 1 & 3 & 5 \\
4 & 3 & 1 & 0 & 5 & 2 & 12 & 6 & 11 & 15 & 14 & 13 & 9 & 10 & 7 & 8 \\
7 & 11 & 6 & 13 & 3 & 14 & 9 & 1 & 10 & 8 & 5 & 4 & 0 & 15 & 12 & 2 \\
\hline
\end{array}$$
\normalsize
$$L_{245}=\tiny \begin{array}{|cccc|cccc|cccc|cccc|}
\hline
0 & 1 & 2 & 3 & 4 & 5 & 6 & 7 & 8 & 9 & 10 & 11 & 12 & 13 & 14 & 15 \\
6 & 10 & 9 & 5 & 13 & 11 & 14 & 12 & 2 & 7 & 4 & 15 & 8 & 0 & 1 & 3 \\
15 & 7 & 14 & 4 & 8 & 10 & 1 & 3 & 5 & 0 & 13 & 12 & 2 & 9 & 11 & 6 \\
8 & 12 & 13 & 11 & \textbf{9} & 0 & 15 & \textbf{2} & 1 & 14 & 6 & 3 & 4 & 7 & 5 & 10 \\
\hline
3 & 0 & 11 & 6 & \textbf{2} & 13 & \textbf{4} & \textbf{14} & 9 & 10 & 1 & 8 & 7 & 5 & 15 & 12 \\
10 & 5 & 4 & 8 & 12 & 6 & 7 & 15 & 14 & 13 & 2 & 0 & 3 & 11 & 9 & 1 \\
14 & 9 & 12 & 1 & 10 & 3 & 8 & 5 & 15 & 4 & 11 & 7 & 6 & 2 & 0 & 13 \\
2 & 13 & 7 & 15 & 1 & 9 & 11 & 0 & 12 & 6 & 3 & 5 & 10 & 14 & 8 & 4 \\
\hline
5 & 6 & 8 & 9 & 7 & 12 & \textbf{2} & \textbf{13} & 0 & 11 & 15 & 10 & 1 & 3 & 4 & 14 \\
13 & 14 & 3 & 2 & 6 & 1 & 0 & 11 & 4 & 5 & 8 & 9 & 15 & 12 & 10 & 7 \\
11 & 4 & 0 & 12 & 15 & 8 & 3 & 10 & 13 & 1 & 7 & 14 & 5 & 6 & 2 & 9 \\
1 & 15 & 10 & 7 & \textbf{14} & 4 & 5 & \textbf{9} & 3 & 2 & 12 & 6 & 11 & 8 & 13 & 0 \\
\hline
12 & 2 & 5 & 14 & 0 & 15 & 10 & 8 & 7 & 3 & 9 & 1 & 13 & 4 & 6 & 11 \\
9 & 8 & 15 & 10 & 11 & 7 & \textbf{13} & \textbf{4} & 6 & 12 & 0 & 2 & 14 & 1 & 3 & 5 \\
4 & 3 & 1 & 0 & 5 & 2 & 12 & 6 & 11 & 15 & 14 & 13 & 9 & 10 & 7 & 8 \\
7 & 11 & 6 & 13 & 3 & 14 & 9 & 1 & 10 & 8 & 5 & 4 & 0 & 15 & 12 & 2 \\
\hline
\end{array}$$
\normalsize
$$L_{247}=\tiny \begin{array}{|cccc|cccc|cccc|cccc|}
\hline
0 & 1 & 2 & 3 & 4 & 5 & 6 & 7 & 8 & 9 & 10 & 11 & 12 & 13 & 14 & 15 \\
6 & 10 & 9 & 5 & 13 & \textbf{14} & \textbf{1} & 12 & 2 & 7 & 4 & 15 & 8 & 0 & \textbf{11} & 3 \\
15 & 7 & 14 & 4 & 8 & 10 & \textbf{11} & 3 & 5 & 0 & 13 & 12 & 2 & 9 & \textbf{1} & 6 \\
8 & 12 & 13 & 11 & 2 & 0 & 15 & 9 & 1 & 14 & 6 & 3 & 4 & 7 & 5 & 10 \\
\hline
3 & 0 & 11 & 6 & 14 & 13 & 2 & 4 & 9 & 10 & 1 & 8 & 7 & 5 & 15 & 12 \\
10 & 5 & 4 & 8 & 12 & 6 & 7 & 15 & 14 & 13 & 2 & 0 & 3 & 11 & 9 & 1 \\
14 & 9 & 12 & 1 & 10 & 3 & 8 & 5 & 15 & 4 & 11 & 7 & 6 & 2 & 0 & 13 \\
2 & 13 & 7 & 15 & 1 & \textbf{11} & \textbf{9} & 0 & 12 & 6 & 3 & 5 & 10 & 14 & 8 & 4 \\
\hline
5 & 6 & 8 & 9 & 7 & 12 & 13 & 2 & 0 & 11 & 15 & 10 & 1 & 3 & 4 & 14 \\
13 & 14 & 3 & 2 & 6 & 1 & 0 & 11 & 4 & 5 & 8 & 9 & 15 & 12 & 10 & 7 \\
11 & 4 & 0 & 12 & 15 & 8 & 3 & 10 & 13 & 1 & 7 & 14 & 5 & 6 & 2 & 9 \\
1 & 15 & 10 & 7 & 9 & 4 & 5 & 14 & 3 & 2 & 12 & 6 & 11 & 8 & 13 & 0 \\
\hline
12 & 2 & 5 & 14 & 0 & 15 & 10 & 8 & 7 & 3 & 9 & 1 & 13 & 4 & 6 & 11 \\
9 & 8 & 15 & 10 & 11 & 7 & 4 & 13 & 6 & 12 & 0 & 2 & 14 & 1 & 3 & 5 \\
4 & 3 & 1 & 0 & 5 & 2 & 12 & 6 & 11 & 15 & 14 & 13 & 9 & 10 & 7 & 8 \\
7 & 11 & 6 & 13 & 3 & \textbf{9} & \textbf{14} & 1 & 10 & 8 & 5 & 4 & 0 & 15 & 12 & 2 \\
\hline
\end{array}$$
\normalsize
$$L_{250}=\tiny \begin{array}{|cccc|cccc|cccc|cccc|}
\hline
0 & 1 & 2 & 3 & 4 & 5 & 6 & 7 & 8 & 9 & 10 & 11 & 12 & 13 & 14 & 15 \\
6 & 10 & 9 & 5 & 13 & 11 & 14 & 12 & 2 & 7 & 4 & 15 & 8 & 0 & 1 & 3 \\
15 & 7 & 14 & 4 & 8 & 10 & 1 & 3 & 5 & 0 & 13 & 12 & 2 & 9 & 11 & 6 \\
8 & 12 & \textbf{11} & \textbf{13} & 2 & 0 & 15 & 9 & 1 & 14 & 6 & 3 & 4 & 7 & 5 & 10 \\
\hline
3 & 0 & \textbf{6} & \textbf{11} & 14 & 13 & 2 & 4 & 9 & 10 & 1 & 8 & 7 & 5 & 15 & 12 \\
10 & 5 & 4 & 8 & 12 & 6 & 7 & 15 & 14 & 13 & 2 & 0 & 3 & 11 & 9 & 1 \\
14 & 9 & 12 & 1 & 10 & 3 & 8 & 5 & 15 & 4 & 11 & 7 & 6 & 2 & 0 & 13 \\
2 & 13 & 7 & 15 & 1 & 9 & 11 & 0 & 12 & 6 & 3 & 5 & 10 & 14 & 8 & 4 \\
\hline
5 & 6 & 8 & 9 & 7 & 12 & 13 & 2 & 0 & 11 & 15 & 10 & 1 & 3 & 4 & 14 \\
13 & 14 & 3 & 2 & 6 & 1 & 0 & 11 & 4 & 5 & 8 & 9 & 15 & 12 & 10 & 7 \\
11 & 4 & 0 & 12 & 15 & 8 & 3 & 10 & 13 & 1 & 7 & 14 & 5 & 6 & 2 & 9 \\
1 & 15 & 10 & 7 & 9 & 4 & 5 & 14 & 3 & 2 & 12 & 6 & 11 & 8 & 13 & 0 \\
\hline
12 & 2 & 5 & 14 & 0 & 15 & 10 & 8 & 7 & 3 & 9 & 1 & 13 & 4 & 6 & 11 \\
9 & 8 & 15 & 10 & 11 & 7 & 4 & 13 & 6 & 12 & 0 & 2 & 14 & 1 & 3 & 5 \\
4 & 3 & 1 & 0 & 5 & 2 & 12 & 6 & 11 & 15 & 14 & 13 & 9 & 10 & 7 & 8 \\
7 & 11 & \textbf{13} & \textbf{6} & 3 & 14 & 9 & 1 & 10 & 8 & 5 & 4 & 0 & 15 & 12 & 2 \\
\hline
\end{array}$$
\normalsize

\section*{Acknowledgements}

Research of Peter Dukes is supported by NSERC Discovery Grant RGPIN-312595-2023.

%% reference citations %%%%%%%%%%%%%%%%%%%%%%%%%%%%%%%%%%%%%%%

\end{document}